\documentclass[12pt]{article}

\usepackage{fancyhdr}
\usepackage[centertags]{amsmath}
\usepackage{amsfonts}
\usepackage{graphicx,psfrag} 
\usepackage{graphics}
\usepackage{graphicx}
\usepackage{wrapfig}
\usepackage{amssymb}
\usepackage{amsthm}
\usepackage{newlfont}
\usepackage{url}
\usepackage{epsfig}
\usepackage{epstopdf}
\usepackage{ifthen}
\usepackage{ifpdf}
\usepackage{rotating}
\usepackage{caption}
\usepackage{amsmath}
\usepackage{subfloat}
\usepackage[dvipsnames]{xcolor}
\usepackage[linesnumbered,lined,boxed,commentsnumbered]{algorithm2e}
\input amssym.def
\input amssym
\input cyracc.def
\usepackage{hyperref}

\numberwithin{equation}{section}

\newtheorem{lemma}{Lemma}[section]

\newtheorem{theorem}{Theorem}[section]
\newtheorem{corollary}{Corollary}[section]

\newtheorem{assumption}{Assumption A}

\newcommand{\be}{\begin{equation}}
\newcommand{\ee}{\end{equation}}
\newcommand{\ba}{\begin{array}}
\newcommand{\ea}{\end{array}}
\newcommand{\bee}{\begin{eqnarray*}}
\newcommand{\eee}{\end{eqnarray*}}
\newcommand{\bea}{\begin{eqnarray}}
\newcommand{\eea}{\end{eqnarray}}

\begin{document}
	
\title{An inexact restoration-nonsmooth algorithm with variable accuracy
for stochastic nonsmooth convex optimization problems in machine learning and stochastic linear
complementarity problems}
\author{Nata\v sa Kreji\' c\footnote{Department of Mathematics and Informatics, Faculty of Sciences, University of Novi
Sad, Trg Dositeja Obradovi\' ca 4, 21000 Novi Sad, Serbia. e-mail: \texttt{natasak@uns.ac.rs}}, 
Nata\v sa Krklec Jerinki\' c\footnote{Department of Mathematics and Informatics, Faculty of Sciences, University of Novi Sad, Trg Dositeja Obradovi\' ca 4, 21000 Novi Sad, Serbia. e-mail: \texttt{natasa.krklec@dmi.uns.ac.rs}}, 
Tijana Ostoji\' c  \footnote{Department of Fundamental Sciences, Faculty of Technical Sciences, University of Novi Sad, Trg Dositeja Obradovi\' ca 6, 21000 Novi Sad, Serbia. e-mail: \texttt{tijana.ostojic@uns.ac.rs}} \footnote{Corresponding author} }

\date{November 2, 2022}
\maketitle
\begin{abstract}
We  study unconstrained optimization problems with nonsmooth and convex objective function in  the form of a mathematical expectation. The proposed method approximates the expected objective function with a sample average function  using Inexact Restoration-based adapted sample sizes. The sample size is chosen in an adaptive manner based on  Inexact Restoration.  The algorithm uses line search and assumes descent directions  with respect to the current approximate function.   We prove the  a.s.  convergence under standard assumptions. 
Numerical results for two types of problems, machine learning loss function for training classifiers and stochastic linear complementarity problems, prove the efficiency of the proposed scheme.  
\end{abstract}


\section{Introduction}
Let us observe an unconstrained optimization problem which objective function takes the form of a mathematical expectation
\begin{equation} 
\min_{x} f(x)=E\left[F(x,\xi)\right],
\label{prob1}
\end{equation}
where $F:\mathbb{R}^n \times \mathbb{R}^m \rightarrow \mathbb{R}$ is continuous and convex function with respect to $x$,   
bounded from below,  $\xi:\Omega\rightarrow \mathbb{R}^m$ is random vector and $(\Omega, \mathcal{F}, P)$  is probability space.  Convexity implies that $ F $ is locally Lipschitz, \cite{BKM14}. No additional smoothness assumption is imposed. A number of important problems can be stated in the form \eqref{prob1} - starting from data analytics with huge data sets which require working with subsamples or online training with the permanently increasing data sets \cite{cevher}, to simulations of natural and industrial processes with number of random parameters \cite{LSOS11, LSOS26, LSOS27, LSOS38}. 

The objective function in \eqref{prob1}  can rarely be computed exactly and  might be nonsmooth. Thus, the main issues that arise in iterative methods for solving \eqref{prob1} are the approximation of the objective function and the choice of search directions. The most common approximation of the mathematical expectation is the Sample Average Approximation (SAA). For a given independent and identically distributed, i.i.d., sample $ \{\xi_1,\ldots,\xi_N\} $ of the size $ N, $ the SAA approximate objective function is defined as 
\begin{equation} 
f_{N}(x)=\frac{1}{N}\sum_{i=1}^{N}f_i(x),
\label{SAA}
\end{equation}
where $ f_i(x)=F(x,\xi_i)$.  The sample vectors  $\xi_1, \ldots, \xi_N$ are assumed to be i.i.d. and the sample size $ N $ determines the accuracy of the approximation \eqref{SAA}, \cite{Shapiro}.  Naturally, larger $ N $ implies higher  accuracy of the approximate function $ f_N,$  but makes any optimization algorithm more costly as the cost of computing $ f_N, $ as well as search directions, increases with $ N. $ There is a vast literature dealing with variable sample size methods for SAA approximations, \cite{BCT, SBNKNKJ, HT2, NKNKJ_varsam, NKJAR}, which range from simple heuristics to complex schemes, all of them with the idea of using cheaper,  lower  accuracy approximations of the objective function whenever possible, in order to save the computational effort. 

The second issue one needs to address is the choice of search directions. In the case of smooth problems we can choose between relatively slow but cheap first order methods or more elaborate and more costly second order methods, depending on a particular problem structure, needed accuracy etc. In the case of nonsmooth problems the gradient is generally replaced by a subgradient or more elaborate schemes like gradient sampling, \cite{grad_sample, kiwiel}, bundle methods, \cite{LV98}, proximal methods, \cite{LS97}, and so on. A number of recent papers deals with second order search directions \cite{asl, asl2, nedic}. 

The method presented in this paper addresses both issues by using an adaptive variable accuracy and descent directions with respect to the current approximate functions. The sample size  is governed by   Inexact Restoration (IR) framework introduced by Martinez and Pilota \cite{JMMP} and consists of two phases: the restoration and the optimality phase. The main idea of IR is to treat  the phases, restoration and optimality, in a modular way and then to use a  merit function, which combines feasibility and optimality  and enforces progress towards a feasible optimal point.  As IR is  constrained optimization tool, the problem \eqref{prob1}
is reformulated  into  a constrained problem as follows 
\begin{equation}
\min f_N(x), \text{ s.t. } f_N(x)=f(x),
\label{prob_constrained}
\end{equation}
where $f$ and $ f_N$ are defined in \eqref{prob1} and \eqref{SAA}, respectively. 

 Notice that \eqref{prob_constrained} is equivalent to \eqref{prob1} if the constraint is satisfied. However if we consider methods that are not strictly feasible, i.e., not all iterations satisfy the constraint, then we can treat $ N $ as an additional variable in the constraint. That is precisely what we will do in the IR approach - in each iteration $ k $ of the method we will determine  a suitable $ N_k. $
There are numerous studies  that have confirmed the benefits of using the IR approach  in the varying accuracy approximations framework, \cite{IRTRUST, NKJMM}. The key advantage of this approach is the fact that feasibility and optimality are kept in balance through merit function. Therefore, the accuracy of the approximate objective function depends on the progress towards optimality in each iteration. Obviously, the accuracy is adaptive, endogenous to the algorithm and there is no need  for  additional parameters or heuristics in the sample size determination.  Furthermore, the sequence of sample sizes is very often nonmonotone, increasing the accuracy (and the computational cost) whenever we approach the solution to ensure good quality of the approximate solution, and decreasing the  accuracy (and the costs) when the current iterate is far away from the solution.  The approach has been used for variable accuracy approximations for the first time in \cite{NKJMM} for the problem of finite sum minimization coupled with  line search descent direction method, based on results from \cite{ana}.
It is extended to trust region framework and constrained problems, \cite{JMMP, IRTRUST, BMK1}. An approach for solving problems with variable accuracy in both objective function and constraints is analyzed in \cite{IRnov2}.

The step size is a challenging issue in stochastic analysis and it was a subject of research in many papers,  \cite{duchi, george, hennig, kingma, review, pesky}. Line search methods, which are an important tool in deterministic optimization, are not easily extended to the stochastic case due to the mutual dependence of  step size  and search direction, which are both random variables in the stochastic framework.   An important study on this topic is given in \cite{review} where the approximations of the objective function and its gradient are assumed to be good enough with a fixed  high probability. Under these settings, the complexity analysis in terms of expected number of iterations to reach near-optimal solution is provided. 
In \cite{LSOS} a second order direction is considered but  an additional sampling is used in Armijo-like condition to overcome the bias issue. The approach presented here differs in several aspects. First of all, we consider the approximate objective of the form (\ref{SAA}) and prove that the algorithm  introduced here yields $ N \to \infty. $ In other words we approach the objective function almost surely under some standard conditions. This property of the algorithm is a direct consequence of IR strategy. Furthermore, the conditional expectation of  the relevant SAA estimator is equal to the objective function under our settings (for details see the final paragraph of Section 2 and   the proof of Lemma \ref{nova}),  and the step size is not directly involved. Another important difference lies in the fact that the objective function and its approximations are not differentiable, and thus the step size analysis is more complicated even in the strongly convex case.

Our contributions are the following. We define Inexact Restoration - Nonsmooth (IR-NS) algorithm for  nonsmooth optimization with variable accuracy and prove  a.s. convergence of the algorithm under the set of standard assumptions.  By using Inexact Restoration for sample size selection we generalize the results from \cite{kineski}. More precisely, since IR-NS pushes the SAA error to zero, in the case of finite sum problems where the objective function is given by \eqref{SAA} with the finite full sample size $N$, the true objective function is reached eventually and the convergence results from  \cite{kineski} hold. IR-NS algorithm  also covers wider class of problems  than finite sums, including infinite sums. 
Our experiments confirm the intuitive reasoning that working with variable, adaptive sample size is more effective than working with predefined or full sample size as in \cite{kineski}. 
To emphasize this fact we present experiments with the same search direction as in \cite{kineski} - the  nonsmooth BFGS descent direction,  and demonstrate the advantages of variable sample size approach proposed in IR-NS. In general,  an  arbitrary descent direction in the sense of Assumption \ref{descent} stated  below  is applicable. 
From theoretical point of view, the complexity  of  order  $\varepsilon^{-2}$ is proved, which also applies to the method from \cite{kineski}. The obtained complexity is in line with the results from \cite{IRnov1} where the complexity of IR is analyzed. The result in \cite{IRnov1} is obtained for smooth constrained problems and is of the form $ \varepsilon_{feas}^{-1} + \varepsilon_{opt}^{-2},$ with $ \varepsilon_{feas} $ being the constant for feasibility and $ \varepsilon_{opt}$ coincides with the $ \varepsilon $  that we consider here. Notice that the problems considered in \cite{IRnov1} are smooth and deterministic. The complexity results obtained in \cite{BMK1} are not comparable to the complexity results for IR-NS  as the methods analyzed in \cite{BMK1} are specialized for smooth problems and problems with regularization.  
It is important to notice that the choice of sample size we propose here introduces stochastic iterative sequence which might seem as an  unnecessary complication if one is dealing with finite sum problems. However we will show that the complexity remains the same and asymptotically we get a.s. convergence, so the stochastic nature does not alter the expected theoretical results. On the other hand,  the intrinsic nature of the sample size variation, based on the progress of the iterative process, yields significant computational cost savings as demonstrated in the numerical results. 

The paper is organized as follows. The algorithm and some preliminaries are given in Section 2, while Section 3 contains convergence analysis. Numerical results are presented in Section 4. Some conclusions are drawn in Section 5.

\section{Proposed algorithm}
The following assumption summarizes the  properties of the problem \eqref{prob1}. 
\begin{assumption}
Assume that $f_i(x)=F(x,\xi_i),$ $i=1,2,\ldots,$ are continuous, convex and bounded from below with a constant $C$ for all $\xi_i.$
\label{pp_Lip}
\end{assumption}
Notice that Assumption A\ref{pp_Lip} implies that $ f $ is convex and continuous function as well as $ f_N. $ Following the standard line search method, we assume that a descent direction can be provided for any given function $f_N$. 
\begin{assumption} \label{descent} For any given $N$, $x$ and $B$ such that $m I\preceq B(x) \preceq M I,$ for some  positive and bounded constants $m\leq M$ we can compute a direction $ p_N \in \mathbb{R}^n $ such that  
$$p_{N}(x)=-B(x) \bar{g}_N(x) \quad \mbox{and}  \sup_{g \in \partial f_N(x)} g^T p_{N}(x) \leq -\frac{m}{2}\|\bar{g}_N(x)\|^2,$$
where $\bar{g}_N(x)  \in \partial f_N(x).$
\label{pkop}
\end{assumption}
Let us briefly discuss the plausibility of the above assumption. One possibility to generate such direction is presented in \cite{kineski} where $ B $ is the BFGS matrix. If an oracle for calculating $  \sup_{g \in \partial f_N(x)} g^T p_{N}(x)$ is available, then we can take the  subgradient descent direction.  Another approach would be to use gradient subsampling techniques \cite{burke}. For directions that satisfy Assumption \ref{descent} the following result holds, \cite{kineski}. We provide the proof for the sake of completeness.

\begin{lemma} 
Let Assumptions A\ref{pp_Lip}  and A\ref{pkop} hold. Then  there exists  $\tau_N(x)>0$ and $ \gamma \in (0,1) $ such that   the subgradient Armijo condition  
\begin{equation*}
 f_N(x+\alpha p_{N}(x)) \leq f_N(x)- \gamma  \alpha \|p_{N}(x)\|^2.
\label{Wolfe1}
\end{equation*}
 holds for all $\alpha \in [0,\tau_N(x)]. $ 
\label{teorema_kineski_armijo}
\end{lemma}

\begin{proof}  Let us fix an arbitrary $N$ and  an arbitrary $x \in \mathbb{R}^n$. If  $\bar{g}_N(x)=0$ the statement is obviously true. In the case    $\bar{g}_N(x)\neq 0 $ we can  define $\delta(\alpha):=f_N(x+\alpha p_{N}(x))$, where   $p_{N}(x)$ is a descent direction satisfying Assumption \ref{pkop}. For such $  p_{N}(x)$  there holds $$\delta'(0)=\sup_{g \in \partial f_N(x)} g^T p_{N}(x)<0.$$ 
Consider   $$l(\alpha):=f_N(x)+\alpha \eta \sup_{g \in \partial f_N(x)} g^T p_{N}(x),$$
for some $\eta \in (0,1)$. Given that $ \sup_{g \in \partial f_N(x)} g^T p_{N}(x) < 0, $ $f_N$ is bounded from below and convex by Assumption A\ref{pp_Lip}, there exists an unique intersection of the functions $\delta$ and $l$ on the interval $\alpha \in (0,\infty)$.  Let us denote this intersection by $\tau_N(x)$. Then, for all $\alpha \in [0,\tau_N(x)]$ there holds 
$$f_N(x+\alpha p_{N}(x))\leq  f_N(x)+\alpha \eta \sup_{g \in \partial f_N(x)} g^T p_{N}(x).$$ Furthermore, Assumption \ref{pkop} implies 
$$f_N(x+\alpha p_{N}(x))\leq f_N(x)-\alpha \eta \frac{m}{2}\|\bar{g}_N(x)\|^2\leq f_N(x)-\alpha \eta \frac{m}{2 M^2}\|p_{N}(x)\|^2$$  and the statement holds for 
$\gamma=\eta m/(2 M^2)$. 
\end{proof}

The problem we are solving is defined by (\ref{prob_constrained}). Clearly the feasibility condition 
$ f_N(x) = f(x) $ can not be enforced in the general case of expected value as in that case we should have $ N \to \infty. $ Furthermore, neither the deviation from feasible condition $ |f(x)-f_N(x)| $ can be computed. Thus we introduce an approximate infeasibility measure as a function $ h(N) $ for arbitrary integer $ N. $ Assume that  $h: \mathbb{N} \to \mathbb{R_+}\cup \{0\}$  is monotonically decreasing function such that $\lim_{N\rightarrow\infty}h(N)= 0.$ In other words, $h(N) $ is a proxy for $ |f(x)-f_N(x)|$.     If  we are solving a finite sum problem, i.e. if $ f(x) = f_{N_{\max}}(x)  $ for a fixed $ N_{\max} $ then for arbitrary $ N \leq N_{\max} $ we can define $h(N)=(N_{max}-N)/N_{max}.$ For the case of unbounded $ N $ one possible simple choice is  $h(N)=N^{-1}. $ 
The merit function  for IR is defined in the usual way 
$$
\Phi(x,N, \theta):= \theta f_{N}(x)+(1-\theta)h(N),
$$
where $\theta \in (0,1)$ is the penalty parameter used to give different weights to the objective function and the measure of infeasibility and $ N $ is an integer that defines the level of accuracy in the approximate function $ f_N. $ 

At each iteration $ k $ we have the accuracy parameter as an integer $ N_k, $ the  solution estimate $ x_k $, the penalty parameter $\theta_k$  and the approximate objective function $ f_{N_k}. $  The algorithm  is as follows. 

\noindent {\bf Algorithm 1}: \texttt{IR-NS} (\texttt{I}nexact \texttt{R}estoration - \texttt{N}on\texttt{s}mooth)
\label{IRBFGS}
\begin{itemize}
\item[\textbf{S0}] Given $ x_0 \in \mathbb{R}^n, N_0 \in \mathbb{N}, \theta_0, r \in (0,1), \beta,  \gamma, \overline{\gamma} > 0$. Set $ k =0$.

\item[\textbf{S1}] Restoration phase. Find $\tilde{N}_{k+1}\geq N_k$ such that
\begin{equation*}
h(\tilde{N}_{k+1})\leq r h(N_k),
\label{Ruslov1NK}
\end{equation*}
\vspace{-0.5cm}
\begin{equation}
f_{\tilde{N}_{k+1}}(x_k)-f_{N_k}(x_k)\leq \beta h(N_k).
\label{Ruslov2NK}
\end{equation}

\item[\textbf{S2}] If
$
\Phi(x_k,\tilde{N}_{k+1},\theta_{k})-\Phi(x_k,N_k,\theta_{k})\leq \frac{1-r}{2}\left( h(\tilde{N}_{k+1})-h(N_k)\right)
$
set $\theta_{k+1} = \theta_{k}.$
Else
$$
\theta_{k+1}:=\frac{(1+r)(h(N_k)-h(\tilde{N}_{k+1}))}{2\left[  f_{\tilde{N}_{k+1}}(x_k)-f_{N_k}(x_k)+ h(N_k)-h(\tilde{N}_{k+1}) \right]}.
$$
\item[\textbf{S3}] Optimization Phase. Choose $N_{k+1}\leq\tilde{N}_{k+1}, $ $ p_{N_{k+1}} \in \mathbb{R}^n $ and $\alpha_k\in(0,1]$ such that 
\begin{equation}
f_{N_{k+1}}(x_k+\alpha_k p_{N_{k+1}}(x_k))-f_{\tilde{N}_{k+1}}(x_k)\leq -\gamma \alpha_k ||p_{N_{k+1}}(x_k)||^2,
\label{S31}
\end{equation}
\begin{equation}
h(N_{k+1})\leq h(\tilde{N}_{k+1})+\bar{\gamma} \alpha_k^2 ||p_{N_{k+1}}(x_k)||^2,
\label{S32}
\end{equation}
\begin{equation}
\Phi(x_k+\alpha_k p_{N_{k+1}}(x_k),N_{k+1},\theta_{k+1})-\Phi(x_k,N_k,\theta_{k+1})\leq \frac{1-r}{2}\bigl( h(\tilde{N}_{k+1})-h(N_k)\bigr).
\label{S3NK}
\end{equation}
\item[\textbf{S4}] Set $p_k=p_{N_{k+1}}(x_k)$,  $x_{k+1}=x_k+\alpha_k p_k$, $k:=k+1$ and go to S1. 
\end{itemize}

Let us briefly discuss the key points of \texttt{IR-NS} algorithm. In Step S1 the feasibility is improved, i.e. a new  sample size candidate $\tilde{N}_{k+1}$ is chosen.  Additionally, the value  $f_{\tilde{N}_{k+1}}(x_k)$ might increase with respect to  $f_{N_k}(x_k)$ by at most $\beta h(N_k).$ Thus, optimality can deteriorate with respect to the previous iteration but the deterioration is controlled by the function $ h, $ i.e., it depends on the accuracy of the objective function. So, for smaller $ N_k $ - which means looser approximation of the true objective function, the deterioration  of optimality can be relatively large, as we assume that we are still far away from solution. Parameter $\beta$ can be arbitrary large, but finite. In some applications (ex. finite sums) one can prove that such $\beta$ exists under standard conditions. However, in general, since we do not impose  differentiability of the objective function nor any other special property,  the following assumption is needed.
 
\begin{assumption}
Suppose that there exists $\beta$ such that \eqref{Ruslov2NK} holds for each $k.$
\label{pp_beta}
\end{assumption}

The penalty parameter is updated in such way that it ensures a decrease of the merit function  as  stated in Lemma \ref{lema31}. Moreover, it can also be shown that the sequence of $\theta_k$ is non-increasing and bounded away from zero which prevents  the  optimality part to vanish from the merit function. The proof of  Lemma \ref{lema31} is fundamentally the same as in \cite[Lemma 2.1]{NKJMM} and thus we omit it here. 

\begin{lemma} \cite{NKJMM}
Let Assumptions A\ref{pp_Lip}- A\ref{pp_beta} hold.   Then the sequence $\lbrace \theta_k \rbrace$ generated by Algorithm \texttt{IR-NS} is positive and non-increasing, the inequality 
$$
\Phi(x_k,\tilde{N}_{k+1},\theta_{k+1})-\Phi(x_k,N_k,\theta_{k+1})\leq \frac{1-r}{2}\left( h(\tilde{N}_{k+1})-h(N_k)\right) 
$$ 
holds and there exists $\theta^*>0$ such that
$\lim_{k\rightarrow\infty}\theta_k=\theta^*.$
\label{lema31}
\end{lemma}

In Step S3 we chose the sample size to be used in the subsequent iteration. Notice that one possible choice is $N_{k+1}=\tilde{N}_{k+1}$ since \eqref{S31}-\eqref{S32} are satisfied due to Lemma \ref{teorema_kineski_armijo} and, as we will prove in Lemma \ref{wd}, there exists $\alpha_k$ which satisfies  inequality \eqref{S3NK}  in that case as well. On the other hand, in order to decrease the overall costs, we try to decrease the sample size if it still provides the  decrease in the merit function \eqref{S3NK}.  The resulting sample size $N_{k+1}$ can be larger, equal or smaller than $N_k$. Our numerical study shows that allowing the decrease of a sample size is beneficial in terms of overall function evaluations. In practical implementations, we estimate the sample size lower bound $N^{trial}_{k+1}$ derived from \eqref{S3NK} and let $N_{k+1} \in \{N^{trial}_{k+1},\lceil(N^{trial}_{k+1}+\tilde{N}_{k+1})/2\rceil, \tilde{N}_{k+1}\}$. We use the backtracking technique for finding $\alpha_k$, but at each backtracking step we try all three candidate values for $N_{k+1}$. This is just one possible approach and the optimal strategy remains an open question, probably problem-dependent. 

\begin{lemma} \label{wd}
Let  Assumptions  A\ref{pp_Lip}- A\ref{pp_beta} hold. Then, there exists $\gamma>0$ such that Step 3 of Algorithm \texttt{IR-NS} is well-defined.
\end{lemma}
\begin{proof}
The algorithm is well defined if there exists a choice of $ N_{k+1}\leq  \tilde{N}_{k+1} $ and a descent direction $ p_k $ such that (\ref{S31}) - (\ref{S3NK}) hold for some $ \alpha_k > 0 $ and a suitable $ \gamma > 0 $ for each $ k. $ Let us take  $N_{k+1}=\tilde{N}_{k+1} $ and retain the same sample so that $f_{N_{k+1}}=f_{\tilde{N}_{k+1}}$. In that case Lemma \ref{teorema_kineski_armijo} implies the existence of  $\tau_k:=\tau_{N_{k+1}}(x_k)>0$ such that the inequality \eqref{S31} holds for all $\alpha \in [0,\tau_k]$. Since \eqref{S32} is trivially satisfied for this choice of $ N_{k+1}$, it remains to prove the existence of $\alpha_k \in [0,\tau_k]$ such that \eqref{S3NK} holds. 
By \eqref{S31}, \eqref{S32} and Lemma \ref{lema31},  for all $\alpha \in [0,\tau_k],$
\begin{align*}
&\Phi(x_{k}+\alpha p_k,N_{k+1},\theta_{k+1})-\Phi(x_k,N_k,\theta_{k+1})\\
& = \Phi(x_{k}+\alpha p_k,N_{k+1},\theta_{k+1})-\Phi(x_k,\tilde{N}_{k+1},\theta_{k+1})+\Phi(x_k,\tilde{N}_{k+1},\theta_{k+1})-\Phi(x_k,N_k,\theta_{k+1})\\ 
& \leq \Phi(x_{k}+\alpha p_k,N_{k+1},\theta_{k+1})-\Phi(x_k,\tilde{N}_{k+1},\theta_{k+1}) + \frac{1-r}{2}\left( h(\tilde{N}_{k+1})-h(N_k)\right)\\
& =\theta_{k+1} \left( f_{N_{k+1}}(x_k+\alpha p_k)-f_{\tilde{N}_{k+1}}(x_k)\right) +\frac{1-r}{2}\left( h(\tilde{N}_{k+1})-h(N_k)\right) \\
& \leq -\theta_{k+1}\gamma\alpha||p_k||^2+\frac{1-r}{2}\left( h(\tilde{N}_{k+1})-h(N_k)\right) \leq \frac{1-r}{2}\left( h(\tilde{N}_{k+1})-h(N_k)\right).
\end{align*}

 Therefore,  \eqref{S3NK} holds  for all $\alpha \in [0,\tau_k]$.
\end{proof}

Notice that in the above Lemma \ref{wd} we proved only that the algorithm is well defined, i.e., we can always take $ N_{k+1}=\tilde{N}_{k+1} $ and the  $ (k+1)$th  iteration is well defined. However, other possibilities for $ N_{k+1} $ exists and we discuss some of them in Section  4. Since the sample size sequence is not monotonically increasing in general, it is not obvious that $ N_k$ tends to infinity.   Nevertheless, using essentially the same proof as in \cite[Theorem 2.1]{NKJMM}, we conclude that infeasibility measure tends to zero yielding the result of $\lim_{k \to \infty}N_k=\infty$. Specially, for the finite sum problem  we conclude that the full sample is reached after a finite number of iterations. The proof of Theorem 2.1 in \cite{NKJMM} contains an important relation stated below
 \begin{equation} \label{h}
 \sum_{k=0}^{\infty} h(N_k) \leq C_1<\infty,
\end{equation}
where $ C_1 > 0 $ is a constant, that we will use in further convergence analysis presented in the next Section.

Let us now provide more insights regarding the stochastic concept of the proposed algorithm. IR-NS yields stochastic sequence of iterates $x_k$. The stochastic nature comes from the sequence of random variables $ N_k $ that determine the samples to be used for the SAA functions. Assume that we are at  iteration $k$ and $x_k$ is known. Denote by $\cal{F}_k$ the $\sigma$-algebra generated by $x_0,...,x_k$, i.e., by random variables that determine $f_{\tilde{N}_{j}}, j=1,...,k$ and  $f_{N_{j}}, j=0,...,k. $  Since the samples are assumed to be i.i.d., we have conditionally unbiased estimators. More precisely, at the beginning of  Step S1 of the algorithm a new sample size $\tilde{N}_{k+1}$ is chosen and  a random sample is generated to obtain $f_{\tilde{N}_{k+1}}$. Thus, since $x_k$ is $\cal{F}_k$-measurable (i.e., known at that point of the algorithmic procedure), there holds 

\begin{equation}
    \label{unbiased} 
    E\left[ f_{\tilde{N}_{k+1}} (x_k) | \cal{F}_k\right]=f(x_k), 
\end{equation}
where $E\left[\cdot | \cal{F}_k\right]$ denotes the conditional expectation with respect to $\cal{F}_k$ [35]. 
Also $ E\left[f_{N_{k+1}} (x_k) | \cal{F}_k\right]=f(x_k)$. However, $ E\left[ f_{N_{k+1}} (x_{k+1}) | \cal{F}_k\right]$ is not equal to $f(x_{k+1})$ in general because $x_{k+1}$ is dependent on $N_{k+1}$. More precisely, the second round of stochastic influence within iteration $k$ comes at the Step S3 where we choose $N_{k+1}$ which may  yield totally different sample for $f_{N_{k+1}}$ with respect to $f_{\tilde{N}_{k+1}} $ in general (each trial sample size may yield different sample). 
Moreover, the direction $p_{N_{k+1}}(x_{k+1})$ and the step size $\alpha_k$ directly depend on the generated samples and thus we lose the martingale property. This is a common situation in stochastic line search (see [14] for instance). In Step S4, we set the next iterate and return to Step S1, repeating the procedure.

\section{Convergence analysis}
The convergence analysis is performed under the set of standard assumptions for stochastic problems stated below. We analyze conditions needed for a.s. convergence of IR-NS and provide complexity result at the end of this section. The two assumptions  stated in this Section are needed to ensure that the  Uniform Law of Large Numbers (ULLN) holds.

\begin{assumption} The objective function $f$ has bounded level sets.
\label{pp_xinD}
\end{assumption}

This assumption holds if the objective function is strongly convex for example, and we have the following result.

\begin{lemma}
 Let  Assumptions  A\ref{pp_Lip}-A\ref{pp_xinD} hold. Suppose that there exists a constant $C_0$ such that $F(x_0,\xi) \leq C_0$ for any $\xi$. Then    $f(x_k) \leq C_2$ holds for all $k$, i.e., $\lbrace x_k\rbrace_{k\in \mathbb{N}} \subseteq D,$ where 
$$D=\{x \in \mathbb{R}^n \; | \; f(x)\leq C_2\}$$
and $C_2=C_0+ 2 \beta C_1$.
\label{nova}
\end{lemma}

\begin{proof} The set $ D $ is compact by Assumption A\ref{pp_xinD}. Using inequalities \eqref{Ruslov2NK}-\eqref{S31}, for all $k$  we obtain 
$$ f_{N_{k+1}}(x_{k+1})\leq f_{\tilde{N}_{k+1}}(x_k) -\gamma \alpha_k ||p_{N_{k+1}}(x_k)||^2\leq f_{N_k}(x_k)+ \beta h(N_k).$$ 
 Furthermore, using the induction argument and (\ref{h}) we get 
\begin{equation*}
\label{rev1a} 
f_{N_{k+1}}(x_{k+1})\leq  f_{N_0}(x_0)+ \beta \sum_{j=0}^{k} h(N_j)\leq f_{N_0}(x_0)+ \beta C_1,
\end{equation*}
for all $k=0,1,...$. 
Obviously, the assumption of uniformly bounded $F$ at the initial point $x_0$ implies that $f_{N_0}(x_0)\leq C_0$ and we obtain \begin{equation} \label{rev1} f_{N_{k}}(x_{k})\leq  C_0 + \beta C_1,\end{equation}
for all $k=1,2,...$.
Finally, by \eqref{unbiased} and  inequalities  \eqref{Ruslov2NK} and \eqref{rev1}  we get 

$$f(x_k)= E\left[f_{\tilde{N}_{k+1}}(x_{k})| {\cal{F}}_{k} \right] \leq E\left[f_{N_k}(x_k)+ \beta h(N_k) | {\cal{F}}_{k} \right]\leq  C_0+2 \beta C_1:=C_2,$$
 which completes the proof. 
\end{proof}

\begin{assumption} 
The function $F$ is dominated by an integrable function on a bounded open set $\tilde{D}^0$ such that  $D \subset \tilde{D}^0.$
\label{pp_domin}
\end{assumption}
Under the stated  assumptions the ULLN \cite{Shapiro} implies  that $ \lim_{N \rightarrow \infty}\sup_{x \in D} |f_{N}(x)-f(x)|=0$ a.s.   
Notice that this equality holds trivially if the sample is finite and the full sample is eventually achieved and retained.
Denote by $X^*=\lbrace x\in \mathbb{R}^n: f(x)=\inf_y f(y):=f^*\rbrace$ the set of solutions for problem  \eqref{prob1}. 
Define 
\begin{equation}
t_k:=\max_{x,y \in \tilde{D}} \lbrace |f(x)-f_{N_{k+1}}(x)|+|f(y)-f_{\tilde{N}_{k+1}}(y)| \rbrace,
\label{tk}
\end{equation}
where $\tilde{D}$ is a compact enlargement of $D,$ i.e., $\tilde{D}$ is the closure of an open set  $\tilde{D}^0 \supset D.$ Therefore, both $ D $ and $ \tilde{D}$ are compact sets and $ D \subsetneq \tilde{D}. $  
Notice that ULLN and the fact $ h(N_k) \to 0 $  imply that $t_k \to 0$  a.s. if $ N_k \to \infty. $ Let us analyse the convergence depending on  properties of the step size sequence $ \{\alpha_k\} $ and the error sequence  $\{t_k\}$.

\begin{theorem}
Let  Assumptions  A\ref{pp_Lip}-A\ref{pp_domin} hold and $\lbrace x_k \rbrace$ be a sequence generated by Algorithm \texttt{IR-NS}. If  $\alpha_k\geq \overline{\alpha}>0$ for all $k \in \mathbb{N}$ then there exists an accumulation point $x^*$ of $\lbrace x_k \rbrace$  which is a solution of problem \eqref{prob1} a.s.
\label{th4.1}
\end{theorem}
\begin{proof} Denote $\bar{g}_k=\bar{g}_{N_k}(x_k).$
Then Assumption A\ref{pkop} and  \eqref{S31} imply 
$$f_{N_{k+1}}(x_{k+1}) \leq f_{\tilde{N}_{k+1}}(x_{k})-\gamma \alpha_k ||p_k||^2 \leq f_{\tilde{N}_{k+1}}(x_{k})-\eta \alpha_k ||\overline{g}_k||^2, $$
where $\eta=\gamma m^2.$ Furthermore,
\begin{align*}
f(x_{k+1})&\leq f_{\tilde{N}_{k+1}}(x_{k})-\eta\alpha_k ||\overline{g}_k||^2+f(x_{k+1})-f_{N_{k+1}}(x_{k+1})\\
&\leq
f(x_k)-\eta\alpha_k ||\overline{g}_k||^2+|f(x_{k+1})-f_{N_{k+1}}(x_{k+1})|+|f_{\tilde{N}_{k+1}}(x_{k})-f(x_{k})|.
\end{align*}
From the definition of $t_k $ \eqref{tk}, we obtain
\begin{equation}
f(x_{k+1})\leq f(x_k)-\eta\bar{\alpha} ||\overline{g}_k||^2+t_k.
\label{nejednakost_sa_tk}
\end{equation}
We will show that 
$\liminf_{k\rightarrow\infty}||\overline{g}_k||^2=0. $
Assume the  contrary, i.e., that  $||\overline{g}_k||^2\geq \varrho>0$ for some $\varrho>0$ and all $ k. $ 
Then
$\eta\overline{\alpha} ||\overline{g}_k||^2\geq \eta \overline{\alpha}\varrho>0.$
Since $t_k \to 0$ a.s., there exists $\overline{k}$ such that for all $k\geq \overline{k}$ there holds 
$t_k\leq \frac{1}{2}\eta \overline{\alpha} ||\overline{g}_k||^2 \text{ a.s.}$ and thus \eqref{nejednakost_sa_tk} implies $f(x_{k+1})\leq f(x_k)-\eta\overline{\alpha}/2$ a.s. Equivalently, for all $s \in \mathbb{N}$ we have 
\begin{equation}
 f(x_{\overline{k}+s})\leq f(x_{\overline{k}})-\frac{s}{2}\eta\overline{\alpha}\varrho \quad \mbox{a.s.}
\label{nej_teorema}
\end{equation}
Letting   $ s \to \infty $  yields a contradiction with the Assumption A\ref{pp_Lip} which implies that $f$ is bounded from bellow. Therefore, we conclude that there there exists $K\subseteq \mathbb{N}$ such that 
$ \lim_{k\in K}\overline{g}_k=0 $ a.s.
Since $\lbrace x_k \rbrace \subset D$  and $D$ is compact there  follows that there exist   $ K_1\subseteq K $ and $x^* \in D$ such that $x^*=\lim_{k\in K_1}x_k.$
Now, using the fact that  $\overline{g}_k \in \partial f_{N_{k+1}}(x_k)$, for all $x \in \mathbb{R}^n$ we have 
$ f_{N_{k+1}}(x)\geq f_{N_{k+1}}(x_k) +\overline{g}^T_k(x-x_k).$
Thus, for  arbitrary $ x \in \tilde{D} $ we have
\begin{align}
f(x)
&\geq
f_{N_{k+1}}(x_k)+\overline{g}^T_k(x-x_k)+f(x)-f_{N_{k+1}}(x) \nonumber\\
&= 
f(x_k)+\overline{g}^T_k(x-x_k)-
\left(f_{N_{k+1}}(x)-f(x)+f(x_k)-f_{N_{k+1}}(x_k)\right)\nonumber\\
&\geq
f(x_k)+\overline{g}^T_k(x-x_k)-\left(|f(x)-f_{N_{k+1}}(x)|+|f(x_k)-f_{N_{k+1}}(x_k)| \right).
\label{deodokaza}
\end{align}
Therefore, $f(x)\geq f(x_k)-||\overline{g}_k|| ||x-x_k||-2 t_k$. Taking the limit over $K_1$ and using the fact that  $||x-x_k||$ is bounded, we obtain that for every $x \in \tilde{D}$ there holds 
\begin{equation}
f(x)\geq f(x^*),   \mbox{ a.s.}
\label{nej_teorema2}
\end{equation} 
Recall that  $x^* \in D$ and $\tilde{D}$ is a compact enlargement of $D$ so $x^*$ cannot be on the boundary of $\tilde{D}$ and there exists $\epsilon>0$ such that  $\mathcal{B} (x^*, \epsilon) \subset \tilde{D}$ and we conclude that $x^*$ is  a local minimizer of $f$ a.s.
Since $f$ is assumed to be convex, we conclude that 
$x^* \in X^*$ a.s. 
\end{proof}

We can also prove that every strictly strong accumulation point \cite{yan} is a solution a.s. A point $x^*$ is called strictly strong accumulation point
of the sequence $\lbrace x_k\rbrace_{k\in \mathbb{N}}$ if there exists a subsequence $K\subseteq \mathbb{N}$ and a constant $b \in \mathbb{N}$ such that $\lim_{k_i\in K} x_{k_i} = x^*$ and $k_{i+1}- k_i \leq b$ for any two consecutive elements $k_i, k_{i+1} \in K.$ According to the available literature, \cite{Shapiro, wardi}, and up to the best of our knowledge, stronger statement in a.s. sense is not possible without some additional assumptions on the rate of increase of $ N_k. $ 
\begin{theorem}
Assume that the conditions of Theorem \ref{th4.1} hold.  Then every strictly strong accumulation point of the sequence $\lbrace x_k \rbrace$ is a solution of problem \eqref{prob1} a.s.
\label{th4.2novo}
\end{theorem}

\begin{proof}

Let $x^*$ be an arbitrary strictly strong accumulation point of the sequence $\lbrace x_k \rbrace$, i.e., 
$x^*=\lim_{i\to\infty} x_{k_i}$
and $s_i:=k_{i+1}-k_i\leq b$ for every $i \in \mathbb{N}.$ 
Since \eqref{nejednakost_sa_tk} holds for each $k\in\mathbb{N},$ we obtain
$$f(x_{k_{i+1}})\leq f(x_{k_i}) -\eta \overline{\alpha}\sum_{j=0}^{s_i-1}||\overline{g}_{k_{i}+j}||^2 +\sum_{j=0}^{s_i-1}t_{k_{i}+j}\leq f(x_{k_i}) -\eta \overline{\alpha}||\overline{g}_{k_i}||^2 +\omega_i, $$
where $\omega_i=\sum_{j=0}^{b-1}t_{k_i+j}.$ Notice that $\omega_i\rightarrow 0, i\rightarrow\infty$  a.s.
We want to show that 
\begin{equation}
\liminf_{i\to\infty} ||\overline{g}_{k_i}||^2=0 \mbox{ a.s.}
\label{liminf_th4.2}
\end{equation}
Assume the contrary, i.e., for all $i\in \mathbb{N}$ there holds $||\overline{g}_{k_i}||^2\geq \varrho>0$ for some $\varrho>0.$
Then,
$\eta\overline{\alpha} ||\overline{g}_{k_i}||^2\geq \eta \overline{\alpha}\varrho>0$ for all $i\in\mathbb{N}.$ Therefore, there exists $\overline{i}$ such that for all $i\geq \overline{i}$ there holds
$\omega_i\leq \frac{1}{2}\eta \overline{\alpha} \varrho $ a.s. 
and thus
$
f(x_{k_{i+1}})\leq f(x_{k_i})-\frac{1}{2}\eta\overline{\alpha} \varrho $ a.s.
Letting  $i\to\infty$ in the last inequality we obtain 
$$f(x^*)\leq f(x^*) -\frac{1}{2}\eta\overline{\alpha} \varrho <  f(x^*),$$
which is contradiction. So, \eqref{liminf_th4.2} holds and repeating the steps \eqref{nej_teorema}-\eqref{nej_teorema2} from the proof of Theorem \ref{th4.1}, we obtain the result, i.e. 
$x^* \in X^*$ a.s.
\end{proof}

Next, we show that the convergence result as in Theorem \ref{th4.1} can be obtained under weaker assumptions on the step size sequence, but assuming  that the sample size $N_k$ is eventually increased fast enough such that  $\sum_{k=0}^{\infty} t_k< \infty$. For instance, if the sample is cumulative, the log bound  given in Proposition 3.5 of  \cite{HT2} holds and $\sum_{k=0}^{\infty} t_k< \infty$ is true if $N_k\geq e^k$. Therefore, one can switch to exponential growth after a certain number of iterations of IR-NS algorithm, taking advantage  of cheap iterations in early stages and theoretically proved convergence for fast increase of the sample size sequence in the later stages of algorithm. The switching point is an interesting problem itself, but beyond the scope of this paper. 

\begin{theorem}
Let  Assumptions  A\ref{pp_Lip}-A\ref{pp_domin} hold and $\lbrace x_k \rbrace$ be a sequence generated by Algorithm \texttt{IR-NS}. If  $\sum_{k=0}^{\infty} \alpha_k=\infty$ and $\sum_{k=0}^{\infty} t_k<\infty$ 
then there exists an accumulation point $x^*$ of $\lbrace x_k \rbrace$  which is a solution of problem \eqref{prob1}.

\label{th4.3}
\end{theorem}

\begin{proof} Following the steps of the proof of Theorem  \ref{th4.1} we obtain $f(x_{k+1})\leq f(x_k)-\eta \alpha_k ||\overline{g}_k||^2+t_k$ for every $k$ and thus 
$$f(x_{k+1}) \leq f(x_0)- \eta \sum_{i=0}^k\alpha_i||\overline{g}_i||^2 +\sum _{i=0}^k t_i. $$ 
The function $ f $ is bounded from below and  $\sum_{k=0}^{\infty} t_k < \infty,$  so we conclude  
\begin{equation} \label{novi1}
\sum_{k=0}^{\infty} \alpha_k ||\overline{g}_k||^2<\infty.
\end{equation} 
 Furthermore, the assumption $\sum_{k=0}^{\infty} \alpha_k=\infty$ implies the existence of a subset $K_1$ such that 
$\lim_{k \in K_1}  \overline{g}_k =0$. Indeed, if we assume the contrary, i.e., that there exists $ \varepsilon > 0 $ such that $\|\overline{g}_k\| \geq \varepsilon >0$ for  $k$ large enough, then we obtain 
$$\sum_{k=0}^{\infty} \alpha_k ||\overline{g}_k||^2\geq \sum_{k=0}^{\infty} \alpha_k \varepsilon^2=\varepsilon^2 \sum_{k=0}^{\infty} \alpha_k=\infty,$$
which is in contradiction with \eqref{novi1}. Since the whole sequence $\lbrace x_k \rbrace_{k \in \mathbb{N}}$ is bounded due to Lemma \ref{nova}, there exist $K_2 \subseteq K_1$ and $x^* \in D$ such that $\lim_{k \in K_2}  x_k =x^*$.  
   Now, repeating the proof of Theorem \ref{th4.1} - the part after  \eqref{nej_teorema},  we conclude that $x^* \in X^*.$  
\end{proof}
The  following result is based on considerations in  \cite{SBNKNKJ} and \cite{grap} and essentially yields  worst-case complexity analysis with respect to the expected  objective function value. 
\begin{theorem}
Let Assumptions  A\ref{pp_Lip}-A\ref{pp_domin} hold, $\varepsilon>0$ and $\lbrace x_k \rbrace$ be a sequence  generated by Algorithm \texttt{IR-NS}.  Furthermore, assume that $\alpha_k\geq \overline{\alpha}>0$ for all $k \in \mathbb{N}$ and $\sum_{k=0}^{\infty} t_k\leq \overline{t}<\infty.$ Then, after at most 
$$\overline{k}=\big\lceil\frac{R^2(\overline{t}+f(x_0)-f^*)}{\eta \overline{\alpha}}\varepsilon^{-2} \big\rceil$$
iterations, we have  $E \left[f(x_{\overline{k}})-f^* \right]\leq \varepsilon, $
where $R$ is the diameter of  $D. $  
\label{th_complexity}
\end{theorem}

\begin{proof} First, notice that \eqref{novi1} holds and since $\alpha_k\geq \overline{\alpha}$ we obtain  $\lim_{k\to\infty}||\overline{g}_k||^2=0.$
Take arbitrary $\varepsilon>0$  and define $\varepsilon_1=\varepsilon/R.$
Since $\overline{g}_k$ tends to zero,  there exists $\overline{k}$ such that 
$||\overline{g}_{\overline{k}}||\leq \varepsilon_1.$
Let $\overline{k}$ be the first such iteration. Then for $k=0,1,\ldots, \overline{k}-1$ we have
$||\overline{g}_k||> \varepsilon_1.$ Moreover, from 
 \eqref{nejednakost_sa_tk} we get  $t_k+f(x_k)-f(x_{k+1})\geq\eta \overline{\alpha}\varepsilon_1^2$ for $k=0,1,\ldots, \overline{k}-1$  and by summing up both sides of this inequality and using $\sum_{k=0}^{\infty} t_k\leq \overline{t}<\infty$ we obtain 
$$\eta \overline{\alpha}\varepsilon_1^2 \overline{k} \leq \overline{t}+f(x_0)-f(x_{\overline{k}})\leq \overline{t}+f(x_0)-f^*,$$
i.e., $\overline{k}\leq (\overline{t}+f(x_0)-f^*)/(\varepsilon_1^2 \eta  \overline{\alpha})=\varepsilon^{-2} (R^2(\overline{t}+f(x_0)-f^*))/(\eta \overline{\alpha}).$
Since $f_{N_{k+1}}$ is convex and $\overline{g}_k\in\partial f_{N_{k+1}}(x_k)$ there holds
$f_{N_{\overline{k}+1}}(x^*)\geq f_{N_{\overline{k}+1}}(x_{\overline{k}})+\overline{g}^T_{\overline{k}}(x^*-x_{\overline{k}}) $, i.e., 
\begin{equation*}
f_{N_{\overline{k}+1}}(x_{\overline{k}})-f_{N_{\overline{k}+1}}(x^*)
\leq \overline{g}^T_{\overline{k}} (x_{\overline{k}}-x^*)
\leq ||\overline{g}_{\overline{k}}||||x^*-x_{\overline{k}}||
\leq \varepsilon_1 R=\varepsilon.
\label{zvezdazvezda}
\end{equation*}
Denote by ${\cal{F}}_{\bar{k}}$ the $\sigma$-algebra generated by $x_0,...,x_{\bar{k}}$. Since the sample is assumed to be i.i.d. and the approximate functions are computed as sample average, we obtain 

$$ E\left[f(x_{\overline{k}})-f(x^*)\right]
=E\left[ E \left[f_{N_{\overline{k}+1}}(x_{\overline{k}})-f_{N_{\overline{k}+1}}(x^*)|{\cal{F}}_{\bar{k}}\right] \right] \leq \varepsilon.$$ 
\end{proof}

Let us conclude this section by considering finite sum case  which falls into the IR-NS framework. Recall  that $h(N_k) \to 0.$ So, in the case of finite sum we have $N_k=N_{\max}$ for all $k\geq k_0$ where $k_0$ is random, but finite.  
 Moreover, $t_k$ becomes zero eventually, so the summability of $t_k$ holds. Furthermore, \eqref{nejednakost_sa_tk} reveals that $f(x_{k+1}) \leq f(x_{k})$ for all $k\geq k_0$ and thus the iterations remain in the level set $\mathcal{L}=\lbrace x | f(x)\leq f(x_{k_0})\rbrace.$ If the level set is compact then the Assumption A\ref{pp_xinD} is obviously satisfied. Finally, notice that Assumption A\ref{pp_domin} is needed only to ensure that $t_k$ tends to zero a.s. which is obviously true in the finite sum case.  Also, notice that in the strongly convex finite sum case there exists $C$ such that all $f_i$ functions are bounded from bellow by $ C. $  Therefore the following result holds.
\begin{corollary}
Let Assumptions  A\ref{pkop}-A\ref{pp_beta} hold and assume  $\sum_k \alpha_k= \infty.$  If $f=f_{N_{max}}$ and $f_i, i=1,...,N_{max}$ are continuous and strongly convex, then there exists an accumulation point $x^*$ of $\lbrace x_k \rbrace$  which is a solution of problem \eqref{prob1}.
Moreover, if   $\alpha_k\geq \overline{\alpha}>0$ for all $k \in \mathbb{N},$ then the worst-case complexity is of order ${\cal{O}}(\varepsilon^{-2})$.
\end{corollary}

\section{Numerical experiments}

In this section, we test \texttt{IR-NS} variable sample size scheme on two classes of nonsmooth convex problems: 1) Finite Sums (FS), i.e., bounded sample size with real-world data,  and 2) Expected Residual Minimization (ERM) reformulation  of  Stochastic Linear Complementarity Problems (SLCP) with   unbounded sample size and  simulated data.  The first class belongs to the machine learning framework and considers $L_2$-regularized binary  hinge loss functions (see \cite{kineski} and the references therein) for binary classification. The considered data sets are given in Table \ref{tabela_dataset} and the problem is of the form 
$$
\min_{x \in \mathbb{R}^n}f(x):=\frac{\lambda}{2}||x||^2+\frac{1}{N_{max}}\sum_{i=1}^{N_{max}} \max(0,1-z_i x^Tw_i),
$$
where $\lambda=10^{-5}$ is a regularization constant, $w_i\in \mathbb{R}^n$ are the input features, $z_i\in \lbrace \pm 1\rbrace$ the corresponding labels, $N_{max}$ is the size of relevant data set (testing or training). 

\begin{table}[h!]
\begin{center}
 \begin{tabular}{||c l c c c c c||} 
 \hline
  & Data set & $N$ & $n$ & $N_{train}$ & $N_{test}$ & $Max_{FEV}$\\ [0.5ex] 
 \hline\hline
 1 & SPLICE \cite{SPLiADL}& 3175& 60 &2540 &635&$10^6$\\ 
 \hline
 2 & MUSHROOMS \cite{MUSH}& 8124& 112 &6500 &1624&$10^6$\\ 
 \hline
 3 & ADULT9 \cite{SPLiADL}& 32561& 123 &  26049&6512&$10^7$\\
 \hline
 4 & MNIST(binary) \cite{MNIST} &70000 & 784 & 60000&10000& $10^7$\\ [0.5ex] 
 \hline
\end{tabular}
\caption{Properties of the data sets used in the experiments.}\label{tabela_dataset}
\end{center}
\end{table}

SLCP consists of finding  a vector $x \in \mathbb{R}^n$ such that
$$
x \geq 0, M(\xi)x + q(\xi)\geq 0, x^T(M(\xi)x + q(\xi)) = 0, \xi \in \Omega,
$$
where $\Omega$ is the underlying sample space,  $M(\xi) \in \mathbb{R}^{n,n}$ is a random matrix and $q(\xi) \in\mathbb{R}^n$ is a random vector.
ERM reformulation (see \cite{NKNKJSR} for example) is defined as follows 
$$
\min f(x) = E\left[||\tilde{F}(x,\xi)||^2\right], \quad \mbox{s. t. } \quad x \geq 0,\
$$
where $\tilde{F}(x, \xi) : \mathbb{R}^n \times \Omega \rightarrow \mathbb{R}^n$, 
$\tilde{F}(x, \xi) = \phi (x, M(\xi)x+q(\xi))  $ 
and $\phi : \mathbb{R}^2 \rightarrow \mathbb{R}$ is  the NCP function defined as $ \phi(a,b) = \min\{a,b\} $.  

The SAA approximate objective function \eqref{SAA} is defined as 
$$
f_{N_k}(x)=\frac{1}{N_k}\sum_{j=1}^{N_k} f_j(x)
$$
with $ f_j(x)=\|\tilde{F}(x,\xi_j)\|^2=\sum_{l=1}^n \left(\min\lbrace x_l,[M(\xi_j)x]_l+[q(\xi_j)]_l\rbrace \right)^2.$


Since numerical results for deterministic (full sample) problem provided in \cite{kineski} reveal the advantages of BFGS-type methods in nonsmooth optimization, we chose to use the method proposed therein for finding a descent direction satisfying Assumption A\ref{pkop}. The functions in consecutive iterations differ in general, and $y_k$ needed for BFGS update is the difference of subgradients of different SAA functions, a safeguard is needed to ensure that the resulting matrices are uniformly positive definite. Thus we start with the identity matrix and skip the BFGS update if $y_k (x_{k+1}-x_k)< 10^{-4} \|y_k\|^2$. Both types of tested problem, FS and ERM allow us to calculate $\sup_{g \in \partial f_{N} (x)} p^T g$ which is crucial for finding the descent BFGS direction.  We denote the proposed algorithm by \texttt{IRBFGS} to emphasize the fact that the BFGS directions are used. 

The parameters of \texttt{IRBGFS} algorithm are  $\theta_0=0.9, \; r=0.95, \overline{\gamma}=1$ and $\gamma=10^{-4}.$   The function $ h $ is defined as  $h(N_k)=(N-N_k)/N$ for FS and $h(N_k)=1/N_{k}$ for ERM problem. Thus, we have $\tilde{N}_{k+1}=\min \{N, \lceil N-r(N-N_k) \rceil \}$ for bounded and $\tilde{N}_{k+1}=\lceil N_k / r \rceil$ for unbounded sample case. $N_0=\lceil 0.1 N \rceil $  for FS, while for ERM problems we take $N_0=1000$. 
 Step S3 is performed as already stated: we estimate the sample size lower bound $N^{trial}_{k+1}$ derived from \eqref{S3NK} and let $N_{k+1} \in \{N^{trial}_{k+1},\lceil(N^{trial}_{k+1}+\tilde{N}_{k+1})/2\rceil, \tilde{N}_{k+1}\}$. The backtracking technique for finding $\alpha_k=0.5^j$ is used, but at each backtracking step we try all three candidate values for $N_{k+1}$. We use cumulative samples, although other approaches are feasible as well. The value $N^{trial}_{k+1}$ is calculated as follows: for FS
$$ N_{k+1}^{trial}:=N_k+\frac{1-r}{2}\cdot\frac{\tilde{N}_{k+1}-N_k}{1-\theta_{k+1}} 
- \hat{\theta}_{k+1}\left(\gamma\alpha||p_{k-1}||^2-f_{\tilde{N}_{k+1}}(x_k)+f_{N_k}(x_{k})\right),$$
where $\hat{\theta}_{k+1}=N\cdot\frac{\theta_{k+1}}{1-\theta_{k+1}}$; for ERM 
 $$ N_{k+1}^{trial}:=\frac{1-\theta_{k+1}}
 {\frac{1-r}{2}\cdot\frac{N_k-\tilde{N}_{k+1}}{\tilde{N}_{k+1}N_k}
 +\frac{1-\theta_{k+1}}{N_{k}}+
 \theta_{k+1}\left(\gamma\alpha||p_{k-1}||^2- f_{\tilde{N}_{k+1}}(x_k)+f_{N_k}(x_{k})  \right)}.$$

The motivation for these choices comes from condition \eqref{S3NK} from Step S3. The merit function at new point should be decreased for at least $\frac{1-r}{2}(h(\tilde{N}_{k+1})-h(N_k))$. Therefore, approximating $||p_k||$ with $||p_{k-1}||$ and using \eqref{S31} and \eqref{S32} from Step S3, we obtain the lower bound $N_{k+1}^{trial}$ for $N_{k+1}$. If this value falls below $N_0$, we simply take $N_{k+1}^{trial}=N_0$. 

Our numerical study has two goals: 
\begin{itemize}
    \item [1)] to investigate if the variable sample size approach is beneficial in terms of overall optimization costs; 
    \item[2)] to investigate if the potential decrease of the sample size coming from S3 is beneficial. 
\end{itemize}
This is why we compare the proposed \texttt{IRBFGS} method to: 
1) \texttt{FBFGS} which takes the full sample (when applicable) at each iteration, i.e., in FS problems $N_k=N_{max}$ for each $k$; 2) \texttt{HBFGS} which takes $N_{k+1}=\tilde{N}_{k+1}$ for each $k$. The criterion for comparison is the number of scalar products denoted by FEV. We report the average values of 10 independent runs. The algorithms  are stopped when the maximum number of scalar products, $Max_{FEV}$ is reached. 
In the FS case, we track the value of the (full sample) objective function, while in the ERM case we track the Euclidean difference between $x_k $ and the solution $x^*$ since the objective function is not computable while the solution is known in advance. 

\begin{figure}[htbp]
    \hspace*{-0.3in}
   \includegraphics[width=0.45\textwidth,angle = 270]{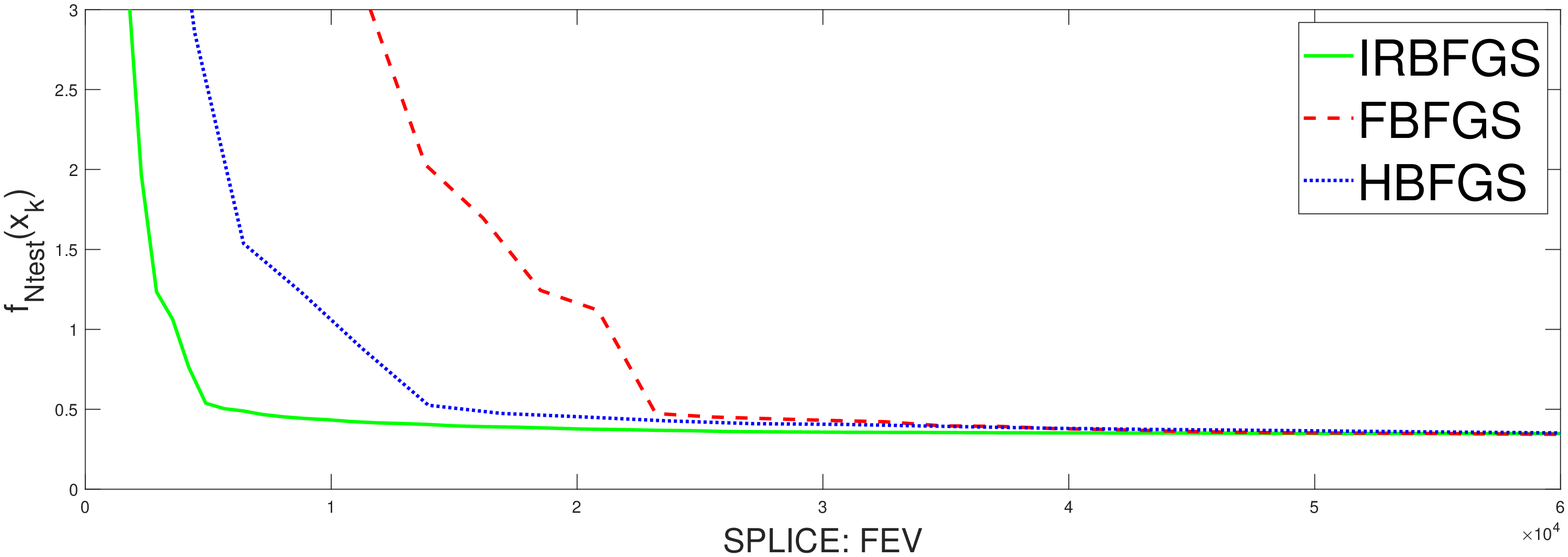}
   \hspace*{-0.3in} 
   \includegraphics[width=0.45\textwidth,angle = 270]{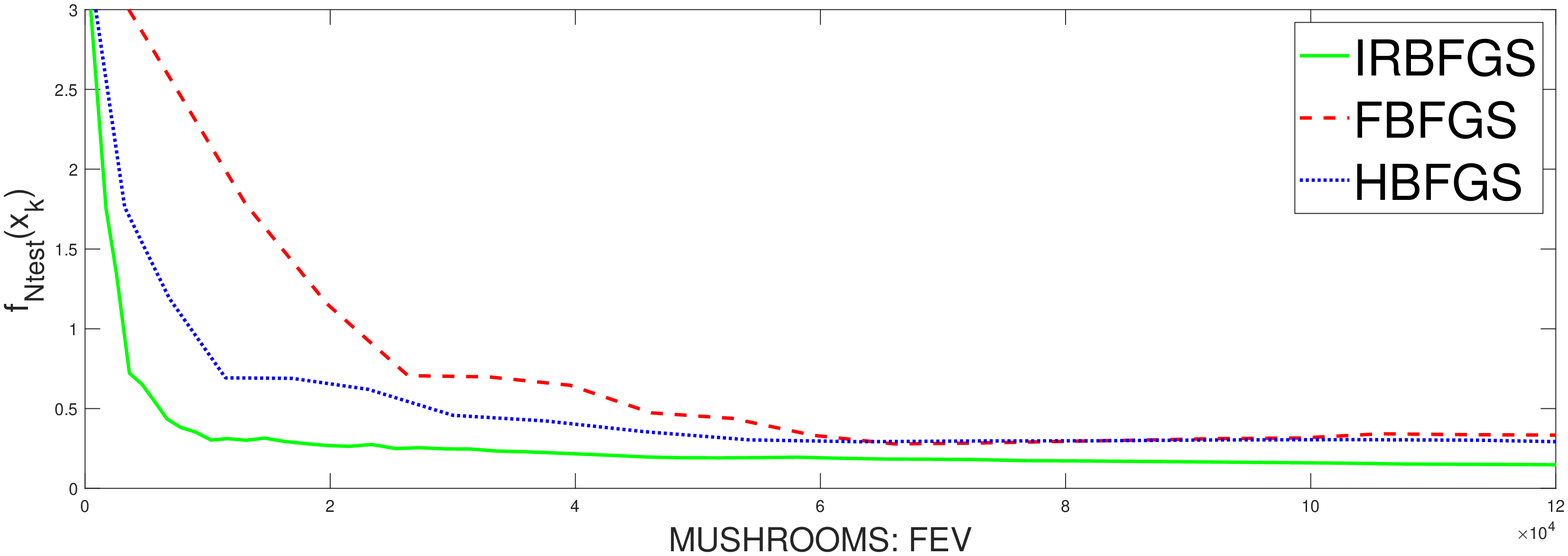}\\
   \hspace*{-0.3in} 
         \includegraphics[width=0.45\textwidth,angle = 270]{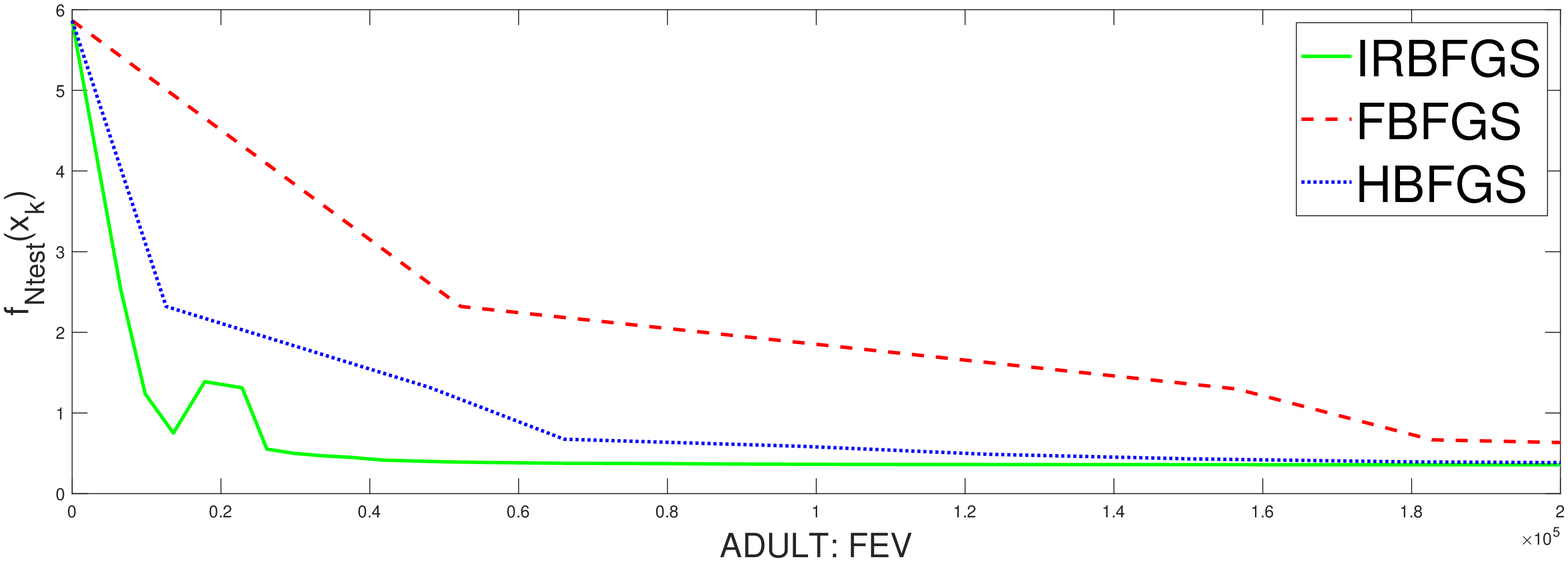}\hspace*{-0.27in} 
         \includegraphics[width=0.45\textwidth,angle = 270]{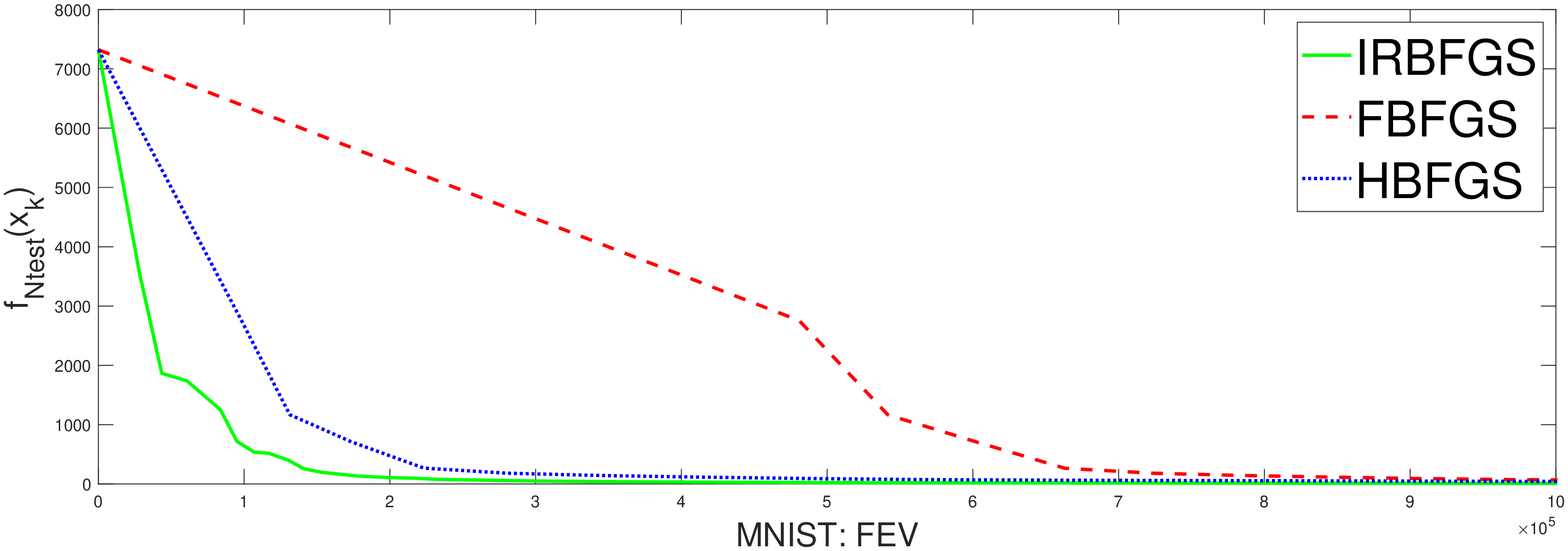}\\
          
    \caption{FS Problem. Testing loss versus function evaluations.}\label{all1}
\end{figure}

\begin{figure}[htbp]
\begin{center}
\includegraphics[width=0.45\textwidth,angle = 270]{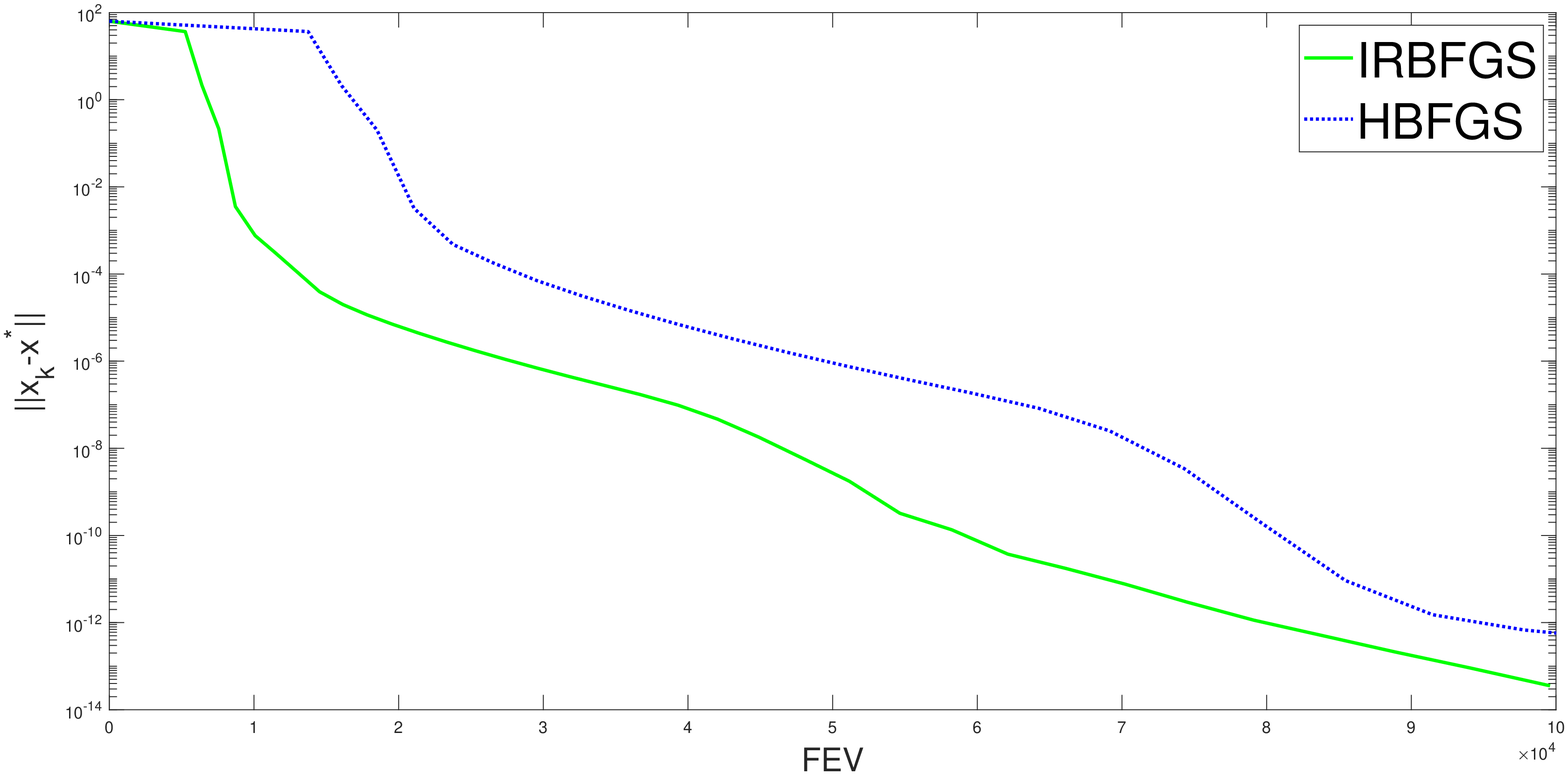}
\caption{ERM Problem. The error $\|x_k-x^*\|$  versus   function evaluations}
\label{all2}
\end{center}
\end{figure}

\begin{figure}[htbp]
    \hspace*{-0.3in}
 \includegraphics[width=0.45\textwidth,angle = 270]{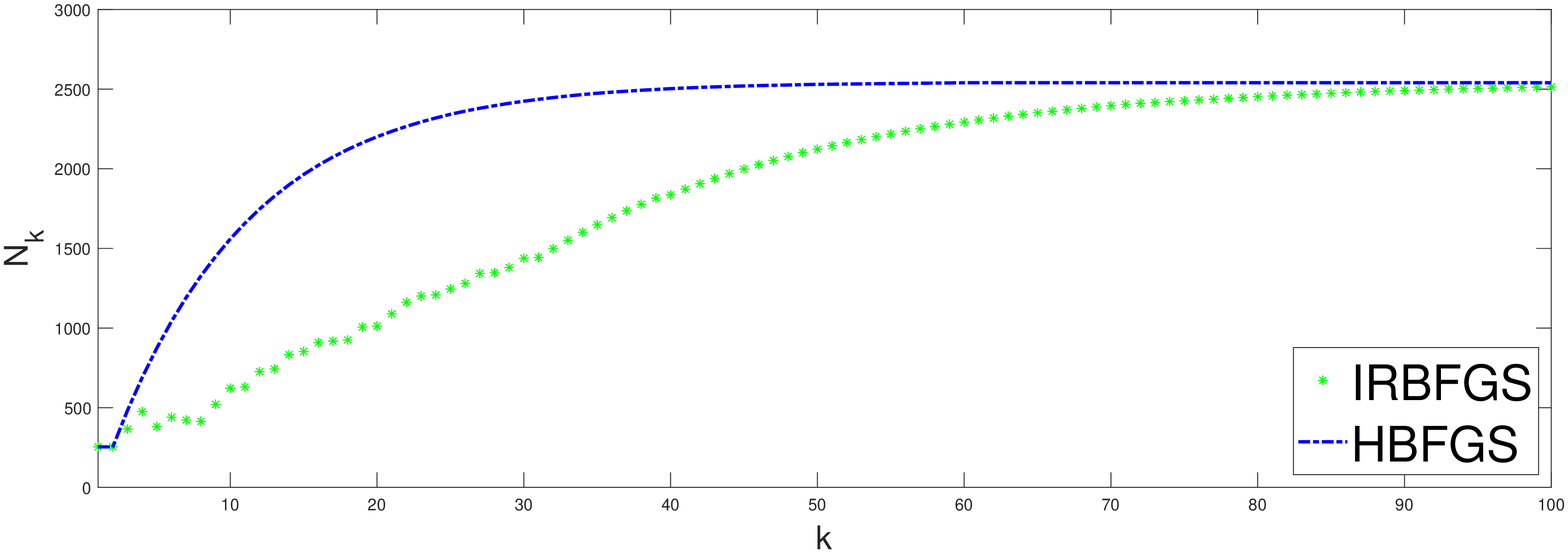} \hspace*{-0.3in}
   \includegraphics[width=0.45\textwidth,angle = 270]{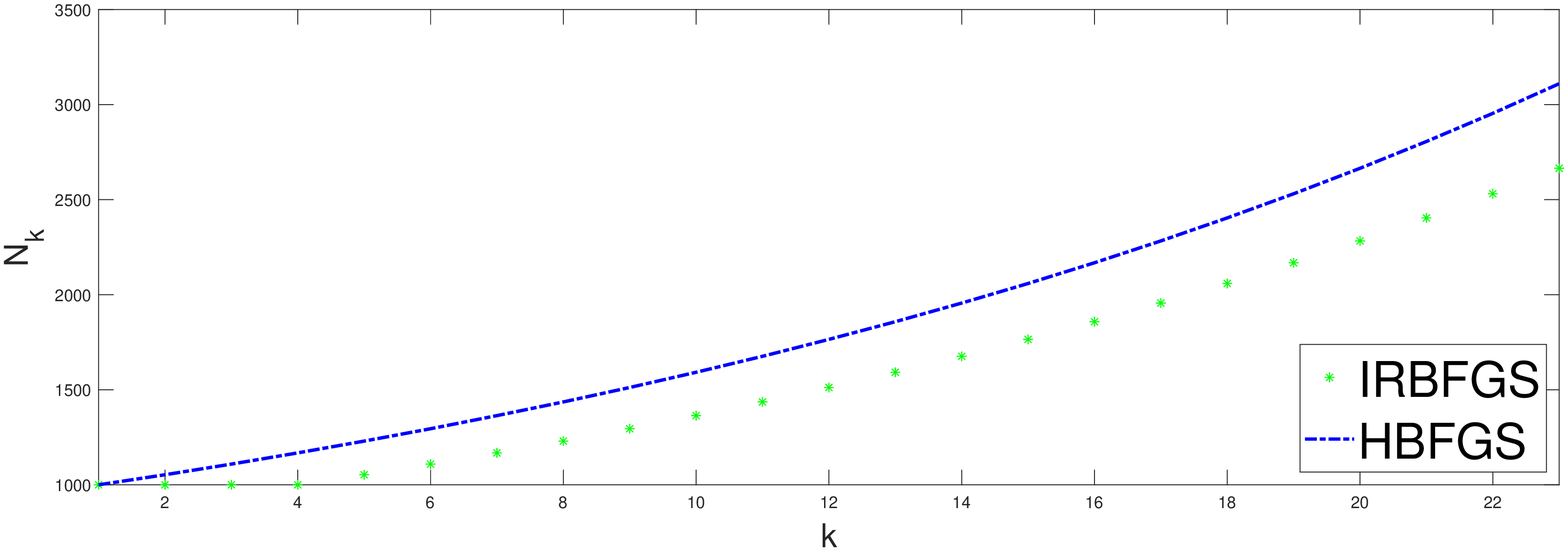}
   \\
   \caption{\texttt{IRBFGS} sample size versus \texttt{HBFGS} sample size sequence: FS Problem - SPLICE data set (left) and ERM Problem (right).}
   \label{SPLICE_uzorak}
\end{figure}

Fig. \ref{all1} shows the results on FS problems with  uniform random $x_0$.  Since training and testing errors behave similarly, we report only the testing error. The $y$-axes are in logarithmic scale. The plots demonstrate the computational savings obtained by \texttt{IRBFGS} in almost all cases. In fact, both subsampled method, \texttt{IRBFGS} and \texttt{HBFGS} use smaller FEV to obtain the solutions of the same quality as the full BFGS - \texttt{FBFGS.} 
Comparing \texttt{IRBFGS} and \texttt{HBFGS}, one can see that \texttt{IRBFGS} is more efficient and occasional decrease of $ N_k $  in Step S3 is beneficial in terms of computational effort measured by FEV. Typical behavior of the sample size sequence is plotted in Fig. \ref{SPLICE_uzorak} (left). 

ERM problems are formed as in \cite{ NKNKJSR, fukushima, li} where the first order methods were tested. Here we proceed with the nonsmooth BFGS direction. We report  the results for problem with $n=100$ and volatility measure $\sigma=10$. $Max_{FEV}$ is set to $10^5$ and the average ending sample size is 4714 for \texttt{IRBFGS} and 3110 for \texttt{HBFGS}. The results and typical behavior of the sample size sequence are presented in 
Fig. \ref{all2} and \ref{SPLICE_uzorak} (right), respectively. As we can see,   \texttt{IRBFGS} algorithm  significantly outperforms the heuristic scheme \texttt{HBFGS.} 

\section{Conclusions} 

We proposed a framework for minimization of nonsmooth convex function in the form of mathematical expectation. The general algorithm is defined within  Inexact Restoration approach, using a suitable approximate function computed as the sample average approximation in each iteration. The sample size is determined adaptively, taking into account the progress toward the stationary point and thus balancing the computational cost and accuracy in endogenous way without heuristic elements. The Armijo line search rule, adapted to the nonsmooth function, is used for step sizes. Algorithm is defined with a general descent direction for nonsmooth function, assuming that a suitable oracle for direction computation is available. It is proved, using the standard IR methodology, that the sample size tends to infinity or attains the fixed maximal value. Therefore, the method generates the approximate solution of desired accuracy but with lower computational costs. The theoretical analysis reveals a.s. convergence towards stationary points under the set of standard assumptions.  The numerical experiments are based on the BFGS direction  adapted to the nonsmooth environment \cite{kineski}. The oracle for computing the direction is taken from literature for the hinge loss problems and Expected Residual Minimization of Stochastic Linear Complementarity Problem. The obtained numerical results are in line with the theoretical considerations and confirm the efficiency of the algorithm.

\medskip
\noindent{\bf Acknowledgement.} We are grateful  to the  editor and the anonymous referees for their comments that helped us to improve the paper. \\
{\bf Funding} The work of Kreji\'c and Krklec Jerinki\'c is supported by Provincial Secretariat for Higher Education and Scientific Research of Vojvodina, grant no. 142-451-2593/2021-01/2. The work of Ostoji\'c is supported by the Ministry of Education, Science and Technological Development, Republic of Serbia. 
\\
{\bf Availability statement} The datasets analysed during the current study are available in the MNIST database
of handwritten digits \cite{MNIST},  LIBSVM Data: Classification (Binary Class) \cite{SPLiADL} and UCI Machine Learning Repository \cite{MUSH}.
\section*{Declarations} 
{\bf Conflict of interest} The authors declare no competing interests.


\begin{thebibliography}{99}
\bibitem{BKM14}{\sc A. Bagirov, N.Karmitsa, M. M{\"a}kel{\"a},}
 Introduction to Nonsmooth Optimization, {\em Springer, (2014), https://doi.org/10.1007/978-3-319-08114-4.}

\bibitem{cevher}{\sc V. Cevher, S. Becker, M. Schmidt,}
Convex Optimization for Big Data: Scalable, randomized, and parallel algorithms for big data analytics, 
{\em IEEE Signal Process. Mag. 31(5) (2014) 32-43, DOI: 10.1109/MSP.2014.2329397.}

\bibitem{LSOS11}{\sc P. J. Carrington, J. Scott, S. Wasserman, eds.,}
Models and Methods in Social Network Analysis, {\em Structural Analysis in the Social Sciences, Cambridge University Press (2005), https://doi.org/10.1017/CBO9780511811395.}

\bibitem{LSOS26}{\sc K. Marti,}
Stochastic optimization methods,
{\em Springer, Heidelberg, third ed. (2015), Applications in engineering and operations research, https://doi.org/10.1007/978-3-662-46214-0.}

\bibitem{LSOS27}{\sc L. Martinez, R. Andrade, E.G. Birgin, J.M. Martinez,}
Packmol: A package for building initial configurations for molecular dynamics simulations, 
{\em J. Comput. Chem. 30 (2009) 2157-2164, https://doi.org/10.1002/jcc.21224.}

\bibitem{LSOS38}{\sc D. Vicari, A. Okada, G. Ragozini, C. Weihs, eds.,}
Analysis and Modeling of Complex Data in Behavioral and Social Sciences, 
{\em Springer, Cham (2014), https://doi.org/10.1007/978-3-319-06692-9.}

\bibitem{Shapiro} {\sc A. Shapiro, D. Dentcheva,  A. Ruszczynski,}
Lectures on stochastic programming: modeling and theory, 
 {\em Society for Industrial and Applied Mathematics (2021), https://doi.org/10.1137/1.9781611976595.}

\bibitem{BCT}{\sc F. Bastin, C. Cirillo, P.L. Toint,} 
An adaptive Monte Carlo algorithm for computing mixed logit estimators, 
{\em Comput. Manag. Sci. 3(1)  (2006) 55-79, https://doi.org/10.1007/s10287-005-0044-y.}

\bibitem{SBNKNKJ}{\sc S. Bellavia, N. Kreji\' c, N. Krklec Jerinki\' c,} Subsampled Inexact Newton methods for minimizing large sums of convex function, 
{\em IMA J. Numer. Anal. 40(4) (2018) 2309-2341, https://doi.org/10.1093/imanum/drz027.}

\bibitem{HT2}{\sc T. Homem-de-Mello,} 
 Variable-Sample Methods for Stochastic Optimization, 
 {\em ACM Trans. Model. Comput. Simul. 13(2) (2003) 108–133, https://doi.org/10.1145/858481.858483.}

\bibitem{NKNKJ_varsam} {\sc N. Kreji\'c, N. Krklec,}
Line search methods with variable sample size for unconstrained optimization,
{\em J. Comput. Appl. Math.  245 (2013) 213-231, https://doi.org/10.1016/j.cam.2012.12.020.}

\bibitem{NKJAR}{\sc N. Krklec Jerinki\' c, A. Ro\v znjik,} Penalty variable sample size method for solving optimization problems with equality constraints in a form of mathematical expectation, 
{\em Numer. Algorithms 83(2) (2020) 701-718, https://doi.org/10.1007/s11075-019-00699-6.} 

\bibitem{grad_sample}{\sc F.E. Curtis, X. Que, }
 An adaptive gradient sampling algorithm for nonsmooth optimization, 
 {\em Optim. Methods Softw. 28(6) (2013) 1302-1324, DOI: 10.1080/10556788.2012.714781.}

\bibitem{kiwiel}{\sc K.C. Kiwiel,}
Convergence of the gradient sampling algorithm for nonsmooth nonconvex optimization,
{\em SIAM J. Optim. 18(2) (2007) 379-388, https://doi.org/10.1137/050639673.}

\bibitem{LV98}{\sc L. Luk\v san, J. Vl\v cek,}
A bundle-Newton method for nonsmooth unconstrained minimization, 
{\em Math. Program. 83(1) (1998) 373 – 391, https://doi.org/10.1007/BF02680566.}

\bibitem{LS97}{ \sc C. Lemarechal, C. Sagastizabal,}
Variable metric bundle methods: From conceptual to implementable forms,
{\em Math. Program. 76(3) (1997) 393 – 410, https://doi.org/10.1007/BF02614390.}

\bibitem{asl}{\sc  A. Asl, M.L. Overton,}
Analysis of limited-memory BFGS on a class of nonsmooth convex functions,
{\em IMA J. Numer. Anal. 41(1) (2021) 1-27, https://doi.org/10.1093/imanum/drz052.}

\bibitem{asl2}{\sc  A. Asl, M.L. Overton,} 
Analysis of the gradient method with an Armijo–Wolfe line search on a class of nonsmooth convex functions, {\em Optim. Methods Softw. 35(2) (2020) 223-242, https://doi.org/10.1080/10556788.2019.1673388.}

\bibitem{nedic}{\sc A. Jalilzadeh, A. Nedi\' c, U. V. Shanbhag, F. Yousefian,}
A Variable Sample-Size Stochastic Quasi-Newton Method for Smooth and Nonsmooth Stochastic Convex Optimization,
{\em Proc. IEEE Conf. Decis. Control, Miami Beach, FL, (2018) 4097-4102, doi: 10.1109/CDC.2018.8619209.}

\bibitem{JMMP}{\sc J.M. Martinez, E.A. Pilotta,}
 Inexact restoration algorithms for constrained optimization, 
 {\em J. Optim. Theory Appl. 104 (2000) 135-163, https://doi.org/10.1023/A:1004632923654.}

\bibitem{IRTRUST}{\sc S. Bellavia, N. Kreji\' c, B. Morini,}
Inexact restoration with subsampled trust-region methods for finite-sum minimization,
{\em  Comput. Optim. Appl. 76 (2020) 701–736, https://doi.org/10.1007/s10589-020-00196-w}

\bibitem{NKJMM} {\sc N. Kreji\' c,  J.M. Martinez,}
 Inexact Restoration approach for minimization with inexact evaluation of the objective function, 
{\em Math. Comput. 85 (2016) 1775-1791, https://doi.org/10.1090/mcom/3025.}

\bibitem{ana}{\sc A. Fischer, A. Friedlander, }
A new line search inexact restoration approach for nonlinear programming,
{\em Comput. Optim. Appl. 46(2) (2010) 333-346, https://doi.org/10.1007/s10589-009-9267-0}

\bibitem{BMK1}{\sc E. Birgin, N. Kreji\' c, J.M. Martinez,} 
Iteration and evaluation complexity on the minimization of functions whose computation is intrinsically inexact, {\em Math. Comput. 89 (2020) 253-278, https://doi.org/10.1090/mcom/3445.} 

\bibitem{IRnov2}{\sc L.F. Bueno, J.M. Martinez,} Inexact Restoration for Minimization with Inexact Evaluation both of the
Objective Function and the Constraints, 
{\em arXiv preprint arXiv:2201.01162, (2022).}

\bibitem{duchi}{\sc J.C. Duchi, E. Hazan, Y. Singer,}
Adaptive subgradient methods for online learning and
stochastic optimization, 
{\em J. Mach. Learn. Res. 12 (2011) 2121-2159.}

\bibitem{george}{\sc A.P. George, W.B. Powell,}
 Adaptive stepsizes for recursive estimation with applications
in approximate dynamic programming, 
{\em Mach. Learn. 65 (2006) 167-198, https://doi.org/10.1007/s10994-006-8365-9.}

\bibitem{hennig}{\sc P. Hennig,}
 Fast probabilistic optimization from noisy gradients, 
 {\em Proc. 30th ICML (2013) 62-70.}

\bibitem{kingma}{\sc D.P. Kingma, J. Ba,}
 Adam: A method for stochastic optimization, 
 {\em Pro. ICLR (2015).}

\bibitem{review}{\sc C. Paquette, K. Scheinberg,}
A stochastic line search method with expected complexity analysis,
{\em SIAM J. Optim. 30(1) (2020) 349-376, https://doi.org/10.1137/18M1216250.}

\bibitem{pesky}{\sc T. Schaul, S. Zhang, Y. LeCun,}
No more pesky learning rates,
{\em PMLR (2013) 343-351.} 

\bibitem{LSOS}{\sc D. di Serafino, N. Kreji\' c, N. Krklec Jerinki\' c, M. Viola,}
LSOS: Line-search Second-Order Stochastic optimization methods for nonconvex finite sums, 
{\em  Math. Comput. (to appear),  arXiv:2007.15966v2 (2021).}

\bibitem{kineski}{\sc J. Yu, S. Vishwanathan, S. Guenter, N. Schraudolph,} 
A Quasi-Newton Approach to Nonsmooth Convex Optimization Problems in Machine Learning,
{\em J. Mach. Learn. Res. 11 (2010) 1145-1200.}

\bibitem{IRnov1}{\sc L.F. Bueno, J.M. Martinez,} 
On the complexity of an inexact restoration method for constrained optimization,
{\em SIAM J. Optim. 30(1) (2020) 80-101, https://doi.org/10.1137/18M1216146}

\bibitem{burke}{\sc J. Burke, A. Lewis,  M. Overton, }  
Approximating subdifferentials by random sampling of gradients,
{\em Math. Oper. Res. 27(3) (2002) 567-584, https://doi.org/10.1287/moor.27.3.567.317.}

\bibitem{yan}{\sc D. Yan, H. Mukai,}
Optimization Algorithm with Probabilistic Estimation,
{\em J. Optim. Theory Appl. 64, 79(2)
(1993) 345-371, https://doi.org/10.1007/BF00940585.}

\bibitem{wardi}{\sc Y. Wardi,}
Stochastic algorithms with Armijo stepsizes for minimization of functions,
{\em J. Optim. Theory Appl. 64 (1990) 399–417, https://doi.org/10.1007/BF00939456.}

\bibitem{grap}{G.N. Grapiglia, E.W. Sachs, }
On the worst-case evaluation complexity of nonmonotone line search algorithms,
{\em Comput. Optim. Appl. 68(3) (2017) 555-577, DOI: 10.1007/s10589-017-9928-3.}

\bibitem{SPLiADL}
{LIBSVM Data: Classification (Binary Class), https://www.csie.ntu.edu.tw/$\sim$cjlin/libsvmtools/datasets/binary.html}

\bibitem{MUSH} 
{\sc M. Lichman,} 
UCI machine learning repository (2013), {\em https://archive.ics.
uci.edu/ml/index.php}

\bibitem{MNIST}
{\sc Y. LeCun, C. Cortes, C.J.C. Burges,}
{The MNIST database of handwritten digits (1998), http://yann.lecun.com/exdb/mnist/}

\bibitem{NKNKJSR} {\sc N. Kreji\' c, N. Krklec Jerinki\' c, S. Rapaji\' c,}
Barzilai-Borwein method with variable sample size for stochastic linear complementarity problems,
{\em Optimization 65(2) (2016) 479-499, https://doi.org/10.1080/02331934.2015.1062008.}

\bibitem{fukushima}
{\sc X. Chen, C. Zhang and M. Fukushima,}
Robust solution of monotone stochastic linear complementarity problems, 
{\em Math. Program. 117(1) (2009) 51-80, https://doi.org/10.1007/s10107-007-0163-z}

\bibitem{li} {\sc X. Li, H. Liu, X. Sun,}
Feasible smooth method based on Barzilai-Borwein method for stochastic linear complementarity problem, 
{\em Numer. Algorithms 57 (2011) 207-215, https://doi.org/10.1007/s11075-010-9424-7}

\end{thebibliography}
\end{document}